\documentclass[draft]{amsart}
\usepackage{amsmath}
\usepackage{amssymb}
\usepackage{graphicx}
\newtheorem{theorem}{Theorem}[section]
\newtheorem{lemma}[theorem]{Lemma}

\theoremstyle{definition}

\newtheorem{remark}[theorem]{Remark}

\numberwithin{equation}{section}

\begin{document}

\title[Asymptotic uniformity of the quantization error for Moran measures]{Asymptotic
uniformity of the quantization error for Moran measures on $\mathbb{R}^1$}
\thanks{The author is supported by National Natural Science Foundation of China No. 11571144.}
\author{Sanguo Zhu}
\address{School of Mathematics and Physics, Jiangsu University
of Technology\\ Changzhou 213001, China.}
\email{sgzhu@jsut.edu.cn}

\begin{abstract}
Let $E$ be a Moran set on $\mathbb{R}^1$ associated with a closed interval $J$ and two sequences $(n_k)_{k=1}^\infty$ and $(\mathcal{C}_k=(c_{k,j})_{j=1}^{n_k})_{k\geq1}$. Let $\mu$ be the infinite product measure (Moran measure) on $E$ associated with a sequence $(\mathcal{P}_k)_{k\geq1}$ of positive probability vectors with $\mathcal{P}_k=(p_{k,j})_{j=1}^{n_k},k\geq 1$. We assume that
\[
\inf_{k\geq1}\min_{1\leq j\leq n_k}c_{k,j}>0,\;\inf_{k\geq1}\min_{1\leq j\leq n_k}p_{k,j}>0.
\]
For every $n\geq 1$, let $\alpha_n$ be an $n$ optimal set in the quantization for $\mu$ of order $r\in(0,\infty)$ and $\{P_a(\alpha_n)\}_{a\in\alpha_n}$ an arbitrary Voronoi partition with respect to $\alpha_n$. For every $a\in\alpha_n$, we write $I_a(\alpha,\mu):=\int_{P_a(\alpha_n)}d(x,\alpha_n)^rd\mu(x)$ and
\[
\underline{J}(\alpha_n,\mu):=\min_{a\in\alpha_n}I_a(\alpha,\mu),\; \overline{J}(\alpha_n,\mu):=\max_{a\in\alpha_n}I_a(\alpha,\mu).
\]
We show that $\underline{J}(\alpha_n,\mu),\overline{J}(\alpha_n,\mu)$ and $e^r_{n,r}(\mu)-e^r_{n+1,r}(\mu)$ are of the same order as $\frac{1}{n}e^r_{n,r}(\mu)$, where $e^r_{n,r}(\mu):=\int d(x,\alpha_n)^rd\mu(x)$ is the $n$th quantization error for $\mu$ of order $r$. In particular, for the class of Moran measures on $\mathbb{R}^1$, our result shows that a weaker version of Gersho's conjecture holds.
\end{abstract}

\keywords {Moran sets, Moran measures, quantization, Voronoi partition, asymptotic
uniformity}
\subjclass[2000]{Primary 28A75, 28A78; Secondary 94A15}
\maketitle
\section{Introduction}
The quantization problem for probability measures has a deep background in information theory and engineering technology such as signal processing and data compression \cite{BW:82,GN:98,Za:63}. Mathematically, this problem consists in the approximation of a given probability measure with discrete probability measures of finite support in the sense of $L_r$-metrics. Asymptotics of the minimum error (quantization error) in the above approximation have been one of the main goals in the study of the quantization problem. In this paper, we study the asymptotic uniformity of the quantization error for Moran measures on the real line. For this class of measures, we will show that, a weaker version of Gersho's conjecture holds. We refer to \cite{GL:00} for the rigorous mathematical foundations for quantization theory and \cite{GPM:02,PG:97} for some significant applications of this theory.
\subsection{Basic definitions and some known results}
Let $\nu$ be a Borel probability measure on $\mathbb{R}^q$. For $x,y\in\mathbb{R}^q$, we denote by $d(x,y)$ the distance between $x$ and $y$ induced by a norm $|\cdot|$ on $\mathbb{R}^q$, and for a subset $\alpha$ of $\mathbb{R}^q$, let $d(x,\alpha):=\inf_{a\in\alpha}d(x,a)$.  Set $\mathcal{D}_{n}:=\{\alpha\subset\mathbb{R}^{q}:1\leq{\rm card}(\alpha)\leq n\}, n\geq 1$. The $n$th quantization error for $\nu$ of order $r\in(0,\infty)$ is given by
\begin{eqnarray}\label{quanerror}
e_{n,r}(\nu):=\bigg(\inf_{\alpha\in\mathcal{D}_{n}}\int d(x,\alpha)^{r}d\nu(x)\bigg)^{1/r}.
\end{eqnarray}
By \cite{GL:00}, $e_{n,r}(\nu)$ equals the minimum error in the approximation of $\nu$ with probability measures supported on at most $n$ points in $L_r$-metrics.

When the infimum in (\ref{quanerror}) is attained at some $\alpha\in\mathcal{D}_n$, we call this $\alpha$ an $n$-optimal set for $\nu$ of order $r$. The set of all $n$-optimal sets for $\nu$ of order $r$ is denoted by $C_{n,r}(\nu)$. It is known \cite{GL:00} that $C_{n,r}(\nu)$ is non-empty whenever the moment condition $\int |x|^rd\nu(x)<\infty$ is satisfied. In particular, for probability measures $\nu$ with compact support, we always have $C_{n,r}(\nu)\neq\emptyset$.

 In the past years, asymptotic properties for the $n$th quantization error, including the upper (lower) quantization dimension and the upper (lower) quantization coefficient, have been intensively studied. Recall that the upper quantization dimension $\overline{D}_{r}(\nu)$ and the lower one $\underline{D}_{r}(\nu)$ for $\nu$ of order $r$ are defined by
\begin{equation*}
\overline{D}_{r}(\nu):=\limsup_{n\to\infty}\frac{\log n}{-\log
e_{n,r}(\nu)};\;\;\underline{D}_{r}(\nu):=\liminf_{n\to\infty}\frac{\log
n}{-\log e_{n,r}(\nu)};
\end{equation*}
and for $s\in(0,\infty)$, the $s$-dimensional upper quantization coefficient $\overline{Q}_r^s(\nu)$ for $\nu$ of order $r$ and the lower one are given by
\[
\overline{Q}_r^{s}(\nu):=\limsup_{n\to\infty}n^{\frac{r}{s}}e^r_{n,r}(\nu),\;
\underline{Q}_r^{s}(\nu):=\liminf_{n\to\infty}n^{\frac{r}{s}}e^r_{n,r}(\nu).
\]
By \cite{GL:00,PK:01}, the upper (lower) quantization dimension is exactly the critical point at which the upper (lower) quantization coefficient jumps from infinity to zero. One may see \cite{GL:00} and \cite{PK:01} for the connections between the upper (lower) quantization dimension (coefficient) and some important objects in fractal geometry.

In \cite{GL:01}, Graf and Luschgy determined the asymptotics for the quantization error for self-similar measures on $\mathbb{R}^q$ with the assumption of the open set condition for the corresponding iterated function system. Let $(f_i)_{i=1}^N$ be a family of contractive similitudes on $\mathbb{R}^q$. By \cite{Hut:81}, there exists a unique non-empty compact set $E$ satisfying $E=\bigcup_{i=1}^Nf_i(E)$.
The set $E$ is referred to as the self-similar set associated with $(f_i)_{i=1}^N$. Given a probability vector $(p_i)_{i=1}^N$ with $p_i>0$ for each $i$, there exists a unique Borel probability measure $P$ supported by $E$ such that
$P=\sum_{i=1}^Np_iP\circ f_i^{-1}$. The measure $P$ is called the self-similar measure associated with $(f_i)_{i=1}^N$ and $(p_i)_{i=1}^N$.
We say that $(f_i)_{i=1}^N$ satisfies the strong separation condition (SSC) if $f_i(E), 1\leq i\leq N$, are pairwise disjoint.
We say that it satisfies the open set condition (OSC) if there exists a bounded non-empty open set $U$ such that $\bigcup_{i=1}^Nf_i(U)\subset U$ and $f_i(U),1\leq i\leq N$, are pairwise disjoint. Let $t_i$ be the contraction ratio of $f_i, 1\leq i\leq N$, and $s_r$ the unique solution of the equation \[
\sum_{i=1}^N(p_is_i^r)^{\frac{s_r}{s_r+r}}=1.
\]
Assuming the OSC for $(f_i)_{i=1}^N$, Graf and Luschgy proved that
\[
\overline{D}_{r}(P)=\underline{D}_{r}(P)=s_r,\;0<\overline{Q}_r^{s}(P)\leq\overline{Q}_r^{s}(P)<\infty.
\]

The above result has a significant influence on later study of the quantization problem, especially for those singular measures supported on fractals. One may also see \cite{GL:00,GL:01,GL:04,GL:08,KZ:16,LM:02,MR:15,PK:01,Zhu:18} for more related work in this direction.

\subsection{Asymptotic uniformity of the quantization error}
A significant concern in quantization theory is, how much contribution each point of an $n$-optimal set make to the $n$th quantization error. This is closely connected with a famous conjecture of Gersho \cite{Ger:79}. In the study of this concern, Voronoi partitions play a crucial role. Recall that a Voronoi partition with respect to a finite set $\alpha\subset\mathbb{R}^q$ means a
partition $\{P_a(\alpha)\}_{a\in\alpha}$ of $\mathbb{R}^q$
satisfying
\begin{eqnarray*}
\{x\in\mathbb{R}^q:\;{\rm d}(x,a)<{\rm
d}\big(x,\alpha\setminus\{a\}\big)\big\}\subset P_a(\alpha)\subset
\big\{x\in\mathbb{R}^q:\;{\rm d}(x,a)={\rm d}(x,\alpha)\big\}.
\end{eqnarray*}
For the above $\alpha\subset\mathbb{R}^q$ and a Borel probability measure $\nu$, we write
\begin{eqnarray*}\label{integral}
&&I(\alpha,\nu):=\int d(x,\alpha)^rd\nu(x),\;\;I_a(\alpha,\nu):=\int_{P_a(\alpha)}d(x,\alpha)^rd\nu(x),\;a\in\alpha;\\
&&\underline{J}(\alpha,\nu):=\min_{a\in\alpha}I_a(\alpha,\nu),\; \overline{J}(\alpha,\nu):=\max_{a\in\alpha}I_a(\alpha,\nu).
\end{eqnarray*}
For each $n\geq 1$, let $\alpha_n$ be an arbitrary $n$-optimal set for $\nu$ of order $r$. A well-known conjecture of Gersho \cite{Ger:79} suggests that
\begin{equation}\label{gersho}
\underline{J}(\alpha_n,\nu),\overline{J}(\alpha_n,\nu)\sim\frac{1}{n}e^r_{n,r}(\nu),
\end{equation}
 where $a_n\sim b_n$ means that $a_n/b_n\to 1$ as $n$ tends to infinity. This conjecture suggests a kind of uniformity:  asymptotically, points of an $n$-optimal set make equal contributions to the $n$th quantization error. Up to now, it has been proved true only for some particular classes of one-dimensional distributions (cf. \cite{GL:12,PG:97}). Recently, Graf, Luschgy and Pag\`{e}s proved in \cite{GL:12} that, for a large class of absolutely continuous measures $\nu$, a weaker version of (\ref{gersho}) holds:
 \begin{equation}\label{weaker}
 \underline{J}(\alpha_n,\nu),\overline{J}(\alpha_n,\nu)\asymp\frac{1}{n}e^r_{n,r}(\nu),
 \end{equation}
 where, for two $\mathbb{R}$-valued variables $X,Y$, $X\asymp Y$ means that there exists a constant $D>0$, such that $DX\leq Y\leq D^{-1}X$. One may see \cite{Kr:13} for some significant results on the above conjecture.

 Assuming the SSC, the author \cite{Zhu:09} established (\ref{weaker}) for self-similar measures on $\mathbb{R}^q$. Without the SSC, it turns out to be rather difficult to examine whether (\ref{weaker}) holds or not. The main obstacle lies in the characterizations for Voronoi partitions with respect to $n$-optimal sets. Due to the lack of "gaps" among cylinder sets, the three-step procedure by means of partitioning, covering and packing, as described in \cite{KZ:15}, is hardly applicable.

 In the present paper, we will prove (\ref{weaker}) for a class of Moran measures on the real line, allowing adjacent cylinder sets touching one another. The advantage lies in the fact that, a bounded interval can always be excluded from its complement by its endpoints. When we add the two endpoints to the quantizing set $\alpha\in\mathcal{D}_n$ and adjust the prospective optimal points inside the interval, its complement would not be affected unfavourably. In spite of this advantage, much work and some new techniques are required for the proof of our main result.

\subsection{Statement of the main result}

 Let $(n_k)_{k=1}^\infty$ be a sequence of integers with $n_k\geq 2$. For $k\geq 1$, let $\mathcal{S}_k=(c_{k,j})_{j=1}^{n_k}$, be a finite sequence of numbers such that
 \[
 \min_{1\leq j\leq n_k}c_{k,j}>0,\;c_{k,1}\cdots+c_{k,n_k}\leq 1.
 \]
We denote by $\theta$ the empty word and set $\Omega_0:=\{\theta\}$. Write
\[
\Omega_k:=\{1,\ldots, n_k\},\;\Omega^*:=\bigcup_{k=0}^\infty\Omega_k,\;\Omega_\infty:=\prod_{k=1}^\infty\Omega_k.
\]
For $\sigma\in\Omega_k$, let $|\sigma|:=k$. For $\sigma=(\sigma_1,\ldots\sigma_{|\sigma|})\in\Omega^*$, we write
\[
\sigma\ast j:=(\sigma_1,\ldots\sigma_{|\sigma|},j),\;1\leq j\leq n_{|\sigma|+1}.
\]
For a word $\sigma=(\sigma_1,\ldots,\sigma_k)\in\Omega_k$ with $1\leq h\leq k$ , we write
\[
|\sigma|=k;\;\sigma|_h=(\sigma_1,\ldots,\sigma_h).
 \]
For $\sigma\in\Omega_\infty$, we set $|\sigma|=\infty$ and define $\sigma|_h$ in the same way as for words in $\Omega^*$.

 Let $J_\theta:=J$ be a closed interval of finite length. We denote by $|A|$ the diameter of a set $A$. For every $k\geq 0$ and $\sigma\in\Omega_k$, let $J_{\sigma\ast j}, 1\leq j\leq n_{k+1}$, be non-overlapping subintervals of $J_\sigma$ with $|J_{\sigma\ast j}|/|J_\sigma|=c_{k+1,j}$ for $1\leq j\leq n_{k+1}$.
 Write
 \begin{equation}\label{moranset}
 E:=\bigcap_{k\geq 1}\bigcup_{\sigma\in\Omega_k}J_\sigma.
 \end{equation}
We call the set $E$ a Moran set associated with $J,(n_k)_{k\geq 1}$ and $(c_{k,j})_{j=1}^{n_k}$. For related results on this type of sets, we refer to \cite{Mu:93,Moran:46,Wen:00} and the references therein.

For each $\sigma\in\Omega^*$, we call $J_\sigma$ a cylinder set of order $|\sigma|$. Without loss of generality, in the following, we assume that $|J|$=1 and set $c_\theta:=1$. Then we have
\begin{eqnarray}\label{productratio}
|J_\sigma|=c_\sigma:=c_{1,\sigma_1}\cdots c_{k,\sigma_k},\;{\rm for}\;\sigma=(\sigma_1,\ldots,\sigma_k)\in\Omega_k,\;k\geq 1.
\end{eqnarray}

 Now let $\Omega_k,k\geq 1$, be endowed with discrete topology and $\Omega_\infty$ be endowed with the corresponding product topology. For every $k\geq 1$, let $(p_{k,j})_{j=1}^{n_k}$ be a positive probability vector and set $p_\theta:=1$. By Kolmogorov consistency theorem, there exists a unique Bore probability measure $\nu$ on $\Omega_\infty$ such that
 \[
\nu([\sigma])=p_\sigma:=p_{1,\sigma_1}p_{2,\sigma_2}\cdots p_{|\sigma|,\sigma_{|\sigma|}},\;{\rm for}\;\sigma=(\sigma_1,\ldots\sigma_{|\sigma|})\in\Omega^*,
 \]
 where $[\sigma]:=\{\tau\in\Omega_\infty:\tau|_{|\sigma|}=\sigma\}$. Let $\pi:\Omega_\infty\to E$ be the natural projection:
 \[
 \pi(\tau):=\bigcap_{k\geq 1}J_{\tau|_k},\;\tau=(\tau_1,\ldots,\tau_k,\ldots)\in\Omega_\infty.
 \]
 It might happen that, for some (at most countably many) points $x\in E$, there correspond two words $\tau,\rho\in\Omega_\infty$ such that $\pi(\tau)=\pi(\omega)=x$. However, one can easily see that the measure $\nu$ is non-atomic, which implies that, for the image measure $\mu:=\nu\circ\pi^{-1}$, we have
 \begin{equation}\label{moranmeasure}
 \mu(J_\sigma)=\nu\circ\pi^{-1}(J_\sigma)=\nu([\sigma])=p_\sigma,\;{\rm for}\;\sigma=(\sigma_1,\ldots\sigma_{|\sigma|})\in\Omega^*.
 \end{equation}
We call the measure $\mu$ a Moran measure on $E$. One may see \cite{Mu:93,Wen:00,Wuli:11} for related results on this type of measures.
 In the present paper, we will prove
\begin{theorem}\label{mthm}
Let $E$ be a Moran set as defined in (\ref{moranset}). Let $\mu$ be the Moran measure satisfying (\ref{moranmeasure}). Assume that
\begin{equation}\label{assume}
\inf_{k\geq1}\min_{1\leq j\leq n_k}c_{k,j}>0,\;\inf_{k\geq1}\min_{1\leq j\leq n_k}p_{k,j}>0.
\end{equation}
 For $n\geq 1$, let $\alpha_n$ be an arbitrary $n$-optimal set for $\mu$ of order $r$ and $\{P_a(\alpha_n)\}_{a\in\alpha_n}$ an arbitrary Voronoi partition with respect to $\alpha_n$. Then we have
\[
\underline{J}(\alpha_n,\mu),\;\overline{J}(\alpha_n,\mu),\;e^r_{n,r}(\mu)-e^r_{n+1,r}(\mu)\asymp\frac{1}{n}e^r_{n,r}(\mu).
\]
\end{theorem}

For the proof of Theorem \ref{mthm}, we will consider some auxiliary measures $\nu_\sigma$ by pushing forward and pulling back the conditional measures of $\mu$ on the cylinder sets $J_\sigma, \sigma\in\Omega^*$. We will establish some properties which hold true uniformly for all $\nu_\sigma,\sigma\in\Omega^*$. These properties will enable us to estimate the number of points of $\alpha_n\in C_{n,r}$ which lies in $J_\sigma$, so that the geometrical size of $P_a(\alpha_n),a\in\alpha_n$, can be addressed. Finally, Theorem 4.1 of \cite{GL:04}, which says that a subset $\beta$ of an $n$-optimal set for $\mu$ is also ${\rm card}(\beta)$-optimal for the conditional measure $\mu(\cdot|\bigcup_{b\in\beta}P_b(\alpha_n))$, will play an important role in our lower estimation for $\underline{J}(\alpha_n,\mu)$. This result, together with the above-mentioned auxiliary measures will enable us to reduce the global optimization problem  for $\mu$ with respect to an arbitrarily large $n$ to local ones for conditional measures $\mu(\cdot|J_\sigma)$ with respect to some bounded numbers.

\section{Preliminaries}

For $\sigma\in\Omega^*$, we write $\sigma^-:=\sigma|_{|\sigma|-1}$ if $|\sigma|>1$ and $\sigma^-=\theta$ if $|\sigma|=1$. We write $\sigma\prec\omega$ if $|\sigma|\leq|\omega|$ and $\sigma=\omega|_{|\sigma|}$. Two words $\sigma,\omega\in\Omega^*$ are called incomparable if neither $\sigma\prec\omega$ nor $\omega\prec\sigma$. A finite subset $\Gamma$ of $\Omega^*$ is called a finite antichain if words in $\Gamma$ are pairwise incomparable; such a $\Gamma$ is said to be a finite maximal antichain if for every $\tau\in\Omega_\infty$, there exists a word $\sigma\in\Gamma$ such that $\sigma\prec\tau$.
We write
\[
\mathcal{E}(\sigma):=p_\sigma s_\sigma^r,\;\sigma\in\Omega^*.
\]

For $k\geq1$, let $(p_{k,j})_{j=1}^{n_k}$ and $(c_{k,j})_{j=1}^{n_k}$, be as given in Theorem \ref{mthm}. We write
\begin{eqnarray*}
&&\underline{p}:=\inf_{k\geq 1}\min_{1\leq j\leq n_k}p_{k,j},\;\;\overline{p}:=\sup_{k\geq 1}\max_{1\leq j\leq n_k}p_{k,j};\\&&\underline{c}:=\inf_{k\geq 1}\min_{1\leq j\leq n_k}c_{k,j},\;\;\overline{c}:=\sup_{k\geq 1}\max_{1\leq j\leq n_k}c_{k,j}.
\end{eqnarray*}
By the assumption (\ref{assume}), we have $\underline{p}>0$ and $\underline{c}>0$.
We define
\[
\eta_r:=\underline{p}\underline{c}^r,\;\Lambda_{k,r}:=\{\sigma\in\Omega^*:\mathcal{E}(\sigma^-)\geq\eta_r^k>\mathcal{E}(\sigma)\},\;\phi_{k,r}:={\rm card}(\Lambda_{k,r}).
\]
One can see that for every $k\geq 1$, $\Lambda_{k,r}$ is a finite maximal antichain. This type of construction was first used in \cite{GL:00} to treat the quantization problem for self-similar measures. The spirit is to seek some kind of uniformity while general probability measures are not uniform. Indeed, for every pair $\sigma,\tau\in\Lambda_{k,r}$, we have
\begin{eqnarray}\label{lambdakrcompare}
\eta_r\mathcal{E}(\tau)\leq\mathcal{E}(\sigma)\leq\eta_r^{-1}\mathcal{E}(\tau),\;\;{\rm implying}\;\;\mathcal{E}(\sigma)\asymp\mathcal{E}(\tau).
\end{eqnarray}
Using the assumption (\ref{assume}) and the arguments in the proof for \cite[Lemma 3.1]{Zhu:09}, one can see that, there exists an integer $M_0$ such that
\begin{equation}\label{lambdacard}
\phi_{k,r}\leq\phi_{k+1,r}\leq M_0\phi_{k,r}.
\end{equation}

\begin{remark}\label{crucial}
Let $\alpha_n\in C_{n,r}(\mu)$ and $\{P_a(\alpha_n)\}_{a\in\alpha_n}$ a Voronoi partition with respect to $\alpha_n$. The following two quantities will be crucial for the characterizations of the geometric structure of each $P_a(\alpha_n)$:
\begin{eqnarray*}
&&{\rm (a).}\;\;{\rm card}\big(\{\sigma\in\Lambda_{k,r}: P_a(\alpha_n)\cap J_\sigma\cap E\neq\emptyset\})\;\;{\rm for}\;\;a\in\alpha_n;\\
&&{\rm (b).}\;\;{\rm card}\big(\{a\in\alpha_n: P_a(\alpha_n)\cap J_\sigma\cap E\neq\emptyset\})\;\;{\rm for}\;\;\sigma\in\Lambda_{k,r}.
\end{eqnarray*}
Once (a) and (b) are well addressed, we will be able to give upper estimates for $\overline{J}(\alpha_n,\mu)$ and $e^r_{n,r}(\mu)-e^r_{n+1,r}(\mu)$ in a convenient manner, while lower estimates will be established by applying \cite[Theorem 4.1]{GL:00} and \cite[Proposition 12.12]{GL:00}.
\end{remark}

 For each $\sigma\in\Omega^*$, let $g_\sigma$ be an arbitrary similitude on $\mathbb{R}^1$ of similarity ratio $c_\sigma$. For the empty word, let $g_\theta$ be the identity map. Let $\mu(\cdot|J_\sigma)$ denote the conditional measure of $\mu$ on $J_\sigma$, namely,
 $ \mu(B|J_\sigma):=p_\sigma^{-1}\mu(B\cap J_\sigma)$ for every  Borel set $B\subset\mathbb{R}^1$. We denote by $\nu_\sigma$ the image measure of $\mu(\cdot|J_\sigma)$:
\begin{equation}\label{nusigma}
\nu_\sigma:=\mu(\cdot|J_\sigma)\circ g_\sigma,\;{\rm implying}\;\;\mu(\cdot|J_\sigma)=\nu_\sigma\circ g_\sigma^{-1}.
\end{equation}
Let $K_\sigma$ denote the support of $\nu_\sigma$. Then we have
\[
K_\sigma\subset g_\sigma^{-1}(J_\sigma)\;\;{\rm and}\;\;|K_\sigma|\leq 1.
\]
\begin{remark}
One can see that $\nu_\sigma$ is an amplification for $\mu(\cdot|J_\sigma)$. It will allow us to connect the integrals over $J_\sigma$ with $\mathcal{E}(\sigma)$, while for suitably chosen $k$ (cf. (\ref{n})) and every $\sigma\in\Lambda_{k,r}$, as we will see, $\mathcal{E}(\sigma)\asymp\frac{1}{n}e^r_{n,r}(\mu)$.
\end{remark}

The subsequent four lemmas will be used to estimate the crucial numbers as mentioned in Remark \ref{crucial}.
\begin{lemma}\label{local}
There exist constants $C,t>0$ such that for every $\sigma\in\Omega^*$, we have
\begin{equation}\label{g2}
\sup_{x\in\mathbb{R}^1}\nu_\sigma(B(x,\epsilon))\leq C\epsilon^t\;{\rm for\;all}\;\;\epsilon>0.
\end{equation}
\end{lemma}
\begin{proof}
By \cite[Lemma 12.3]{GL:00}, it suffices to show (\ref{g2}) for $x\in K_\sigma$ and $\epsilon\in(0,\underline{c})$. Set
\begin{eqnarray*}
&\mathcal{A}_1(\sigma,\epsilon):=\{\tau\in\bigcup_{k=1}^\infty\prod_{h=|\sigma|+1}^{|\sigma|+k}\Omega_h: c_{\tau^-}\geq\epsilon>c_\tau\},
\\&\mathcal{A}_2(\sigma,\epsilon):=\{\tau\in\mathcal{A}_1(\sigma,\epsilon):g_\sigma^{-1}(J_{\sigma\ast\tau^-})\cap B(x,\epsilon)\neq\emptyset\},
\end{eqnarray*}
where $c_\tau$ is defined in the same way as we did for words in $\Omega^*$ (cf. (\ref{productratio})). Then one can see that ${\rm card}(\mathcal{A}_2(\sigma,\epsilon))\leq 4$ and $|\tau|\geq\log\epsilon/\log\underline{c}$. It follows that
\[
\nu_\sigma(B(x,\epsilon))\leq 4\overline{p}^{\frac{\log\epsilon}{\log\underline{c}}-1}=4\overline{p}^{-1}\epsilon^{\frac{\log\overline{p}}{\log\underline{c}}}.
\]
This and \cite[Lemma 12.3]{GL:00} complete the proof of the lemma.
\end{proof}

Our next lemma is an analogue of \cite[Lemma 5.8]{GL:04} by Graf and Luschgy.
\begin{lemma}\label{glanalogue}
Let $\nu$ be a probability measure $\mathbb{R}^q$ with compact support $K_\nu$. Then
\[
e^r_{l,r}(\nu)-e^r_{l+1,r}(\nu)\leq 3^r|K_\nu|^r(l+1)^{-1},\;l\geq 1.
\]
\end{lemma}
\begin{proof}
Let $\beta_{l+1}\in C_{l+1,r}(\nu)$ and let $\{P_b(\beta_{l+1})\}_{b\in\beta_{l+1}}$ be a Voronoi partition with respect to $\beta_{l+1}$. There exists some $b_0\in\beta_{l+1}$ with $\nu(P_{b_0}(\beta_{l+1}))\leq(l+1)^{-1}$. We set $\gamma:=\beta_{l+1}\setminus\{b_0\}$. Then we have $d(x,\gamma)=d(x,\beta_{l+1})$ for all $x\in K_\nu\setminus P_{b_0}(\beta_{l+1})$. Note that $\sup_{x\in K_\nu}d(x,\beta_{l+1})\leq 2|K_\nu|$. Fix an arbitrary $b\in\gamma$ and $y\in P_b(\beta_{l+1})\cap K_\nu$. Then for $x\in P_{b_0}(\beta_{l+1})\cap K_\nu$, we have
\[
d(x,\gamma)\leq d(x,y)+d(y,b)\leq3|K_\nu|.
\]
Note that $e^r_{l,r}(\nu)\leq I(\gamma,\nu)$. It follows that
\begin{eqnarray*}
e^r_{l,r}(\nu)-e^r_{l+1,r}(\nu)&\leq& I(\gamma,\nu)-I(\beta_{l+1},\nu)\\&=&\int_{P_{b_0}(\beta_{l+1})}d(x,\gamma)^rd\nu(x)-\int_{P_{b_0}(\beta_{l+1})}d(x,\beta_{l+1})^rd\nu(x)\\&\leq& \int_{P_{b_0}(\beta_{l+1})}d(x,\gamma)^rd\nu(x)\leq 3^r|K_\nu|^r(l+1)^{-1}.
\end{eqnarray*}
This completes the proof of the lemma.
\end{proof}

\begin{lemma}\label{zhu1}
Let $\nu$ be a probability measure $\mathbb{R}^q$ with compact support $K_\nu$. Assume that $|K_\nu|\leq 1$ and for some constants $D$ and $s$, $\sup_{x\in\mathbb{R}^q}\nu(x,\epsilon)\leq D\epsilon^s$ for all $\epsilon>0$. Then there exists a number $\zeta_{l,r}>0$ depending on $l,D$ and $s$ such that
\[
e^r_{l,r}(\nu)-e^r_{l+1,r}(\nu)\geq \zeta_{l,r}.
\]
\end{lemma}
\begin{proof}
This is Lemma 2.3 of \cite{Zhu:17}. See Lemma \ref{pre06a} below for the basic idea.
\end{proof}

Let $\mu$ be as defined in (\ref{moranmeasure}) and $\beta\subset\mathbb{R}^1$. Let $l\geq 1$ and $L\geq 1$. We write
\[
I_\rho(\beta,\mu):=\int_{J_\rho} d(x,\beta)^rd\mu(x),\;\rho\in\Omega^*;\;\;\Psi_{l,L}:=\prod_{h=l+1}^{l+L}\Omega_h.
\]

 Using Lemmas \ref{glanalogue} and \ref{zhu1}, we are able to choose some constants which will be used in the characterization for the optimal sets. We have

\begin{lemma}\label{pre01}
For every $\sigma\in\Omega^*$, let $\nu_\sigma$ be as defined in (\ref{nusigma}). Then
\begin{enumerate}
 \item[\rm (1)] There exists an integer $M_1\geq4$, such that for $\sigma\in\Omega^*,\omega\in\Psi_{|\sigma|,3}$ and $\beta\subset J_{\sigma\ast\omega}^c$,
 \[
 I_{\sigma\ast\omega}(\beta,\mu)>\mathcal{E}(\sigma)e^r_{M_1-3,r}(\nu_\sigma).
 \]
\item[\rm (2)]There exists an integer $M_2>M_1+4$, such that, for all $l\geq M_2$ and $\sigma,\tau\in\Omega^*$,
\[
e^r_{l-M_1-4,r}(\nu_\sigma)-e^r_{l+2,r}(\nu_\sigma)<\eta_r\big(e^r_{M_1+1,r}(\nu_\tau)-e^r_{M_1+2,r}(\nu_\tau)\big).
\]
\item[\rm (3)] Let $M_0$ be the same as in (\ref{lambdacard}). There exists an integer $M_3>M_0(M_2+2)+7$, such that, for all $l\geq M_3$ and every pair $\sigma,\tau\in\Omega^*$,
\[
e^r_{l-7,r}(\nu_\sigma)-e^r_{l+2,r}(\nu_\sigma)<\eta_r\min_{1\leq h\leq M_0(M_2+4)+2}(e^r_{h,r}(\nu_\tau)-e^r_{h+1,r}(\nu_\tau)).
\]
\end{enumerate}
\end{lemma}
\begin{proof}
(1) Let $\sigma\in\Omega^*$ and $\omega\in\Psi_{|\sigma|,3}$ be given. By the construction of $E$, there exists a $\tau_0\in\Psi_{|\sigma|+5}$ such that
\[
d(J_{\sigma\ast\omega\ast\tau_0},J_{\sigma\ast\omega}^c)\geq\underline{c}^5|J_\sigma|=\underline{c}^5c_\sigma,\;
\mu(J_{\sigma\ast\omega\ast\tau_0})=p_{\sigma\ast\omega\ast\tau_0}\geq p_\sigma\underline{p}^5.
\]
Hence, for $\beta\subset J_{\sigma\ast\omega}^c$, we have $g_\sigma^{-1}(\beta)\subset(g_\sigma^{-1}(J_{\sigma\ast\omega}))^c$ and
\[
d(g_\sigma^{-1}(J_{\sigma\ast\omega\ast\tau_0}),(g_\sigma^{-1}(J_{\sigma\ast\omega}))^c)\geq \underline{c}^5.
\]
Using this and the similarity of $g_\sigma$, we deduce
\begin{eqnarray*}\label{g3}
I_{\sigma\ast\omega}(\beta,\mu)&=&\int_{J_{\sigma\ast\omega}} d(x,\beta)^rd\mu(x)\nonumber\\
&\geq&\mu(J_\sigma)\int_{J_{\sigma\ast\omega\ast\tau_0}} d(x,\beta)^rd\mu(\cdot|J_\sigma)(x)
\nonumber\\&=&p_\sigma\int_{J_{\sigma\ast\omega\ast\tau_0}} d(x,\beta)^rd\nu_\sigma\circ g_\sigma^{-1}(x)\nonumber\\
&=&p_\sigma\int_{g_\sigma^{-1}(J_{\sigma\ast\omega\ast\tau_0})} d(g_\sigma(x),\beta)^rd\nu_\sigma(x)\nonumber\\
&=&\mathcal{E}(\sigma)\int_{g_\sigma^{-1}(J_{\sigma\ast\omega\ast\tau_0})} d(x,g_\sigma^{-1}(\beta))^rd\nu_\sigma(x)\nonumber\\&\geq&
\mathcal{E}(\sigma)\underline{p}^5\underline{c}^{5r}=:\mathcal{E}(\sigma)C_0.
\end{eqnarray*}
For $x\in\mathbb{R}$, let $[x]:=\max\{h\in\mathbb{N}:h\leq x\}$. Note that $|K_\sigma|\leq 1$. By Lemma \ref{glanalogue}, (1) follows by setting $M_1:=[3^rC_0^{-1}]+4$.

(2) Let $\sigma,\tau\in\Omega^*$. Note that $|K_\tau|\leq 1$. By Lemma \ref{zhu1}, for $M_1$ as chosen in (1), there exists a number $\zeta_{M_1,r}>0$ depending on $C$ and $t$, such that
\[
e_{M_1+1,r}^r(\nu_\tau)-e_{M_1+2,r}^r(\nu_\tau)>\zeta_{M_1,r}.
\]
Hence, by Lemma \ref{glanalogue}, for $A:=3^r(M_1+6)([\zeta_{M_1,r}^{-1}\eta_r^{-1}]+1)$, and $l\geq A$,
\begin{eqnarray*}
e^r_{l,r}(\nu_\sigma)-e^r_{l+M_1+6,r}(\nu_\sigma)&<&\frac{3^r(M_1+6)}{l+1}<\eta_r\zeta_{M_1,r}
\\&<&\eta_r\big(e^r_{M_1+1,r}(\nu_\tau)-e^r_{M_1+2,r}(\nu_\tau)\big).
\end{eqnarray*}
It suffices to set $M_2:=[A]+M_1+4$.

(3) This can be seen in the same manner as we did for (2).
\end{proof}

\section{A characterization of the $n$-optimal sets for $\mu$}
Let $M_i, i=1,2,3$, be the integers as chosen in section 2. For every $n\geq (M_2+2)\phi_{1,r}$, there exists a unique integer $k$ such that
\begin{eqnarray}\label{n}
(M_2+2)\phi_{k,r}\leq n<(M_2+2)\phi_{k+1,r}.
\end{eqnarray}
In the remaining part of the paper, we always assume that $n\geq (M_2+2)\phi_{1,r}$ and let $k$ be the integer as chosen in (\ref{n}). In this section, we are devoted to upper and lower estimates for the numbers ${\rm card}(\alpha_n\cap J_\sigma)$ with $\sigma\in\Lambda_{k,r}$. These numbers will enable us to address (a) and (b) in Remark \ref{crucial}.
For $\sigma\in\Omega^*$, we denote by $x_1(\sigma),x_2(\sigma)$ the left and right endpoint of $J_\sigma$, namely,
\begin{equation}\label{endpoints}
x_1(\sigma):=\min_{x\in J_\sigma}x,\;\;x_2(\sigma):=\max_{x\in J_\sigma}x.
\end{equation}

\begin{remark}\label{remgl}
We will need the following facts.
\begin{enumerate}
\item[(i)]Note that ${\rm card}({\rm supp}(\mu))=\infty$. By Theorem 4.1 of \cite{GL:00}, we have
\[
e_{h+1,r}(\mu)<e_{h,r}(\mu)\;\;{\rm for\;all}\;\;h\geq 1.
\]
\item[(ii)] Let $h\in\mathbb{N}$ and $\sigma\in\Omega^*$. For $\gamma\in C_{h,r}(\mu(\cdot|J_\sigma))$, by the similarity of $g_\sigma$ and Lemma 3.2 in \cite{GL:00}, we have $g_\sigma^{-1}(\gamma)\in C_{h,r}(\nu_\sigma)$. Hence,
\begin{equation*}
I_\sigma(\gamma,\mu)=\int_{J_\sigma}d(x,\gamma)^rd\mu(x)=\mathcal{E}(\sigma)\int d(x,g_\sigma^{-1}(\gamma))^rd\nu_\sigma(x)=\mathcal{E}(\sigma)e^r_{h,r}(\nu_\sigma).
\end{equation*}
\item[(iii)] Let $\emptyset\neq\beta\subset\mathbb{R}^1$ be a finite set. For $\sigma\in\Lambda_{k,r}$, we set $\widetilde{N}_\sigma:={\rm card}(\beta\cap J_\sigma)$ and
    \[
    \gamma:=(\beta\cap J_\sigma)\cup\{x_1(\sigma),x_2(\sigma)\}.
    \]
Then we have $d(x,\gamma)\leq d(x,\beta)$ for all $x\in J_\sigma$. If $\widetilde{N}_\sigma\geq 1$, then we have
\begin{equation*}
I_\sigma(\beta,\mu)\geq I_\sigma(\gamma,\mu)=\mathcal{E}(\sigma)\int d(x,g_\sigma^{-1}(\gamma))^rd\nu_\sigma(x)\geq\mathcal{E}(\sigma)e^r_{\widetilde{N}_\sigma+2,r}(\nu_\sigma).
\end{equation*}
If $\widetilde{N}_\sigma=0$, then we have $\beta\subset J_\sigma^c$. By Lemma \ref{pre01} (1), we have
\[
I_\sigma(\beta,\mu)\geq\mathcal{E}(\sigma)e^r_{M_1-3,r}(\nu_\sigma)>\mathcal{E}(\sigma)e^r_{M_1+1,r}(\nu_\sigma).
\]
\end{enumerate}
\end{remark}

Let $n$ and $k$ satisfy (\ref{n}). We fix an arbitrary $\alpha_n\in C_{n,r}(\mu)$ and write
\[
L_\sigma:={\rm card}(\alpha_n\cap J_\sigma),\;\sigma\in\Lambda_{k,r}.
\]
With the next lemma, we give a lower estimate for $L_\sigma$. That is,
\begin{lemma}\label{pre02}
For every $\sigma\in\Lambda_{k,r}$, we have $L_\sigma\geq M_1$.
\end{lemma}
\begin{proof}
 Note that we have altogether $\phi_{k,r}$ intervals $J_\sigma$ with $\sigma\in\Lambda_{k,r}$. Write
 \[
 x_1:=\min\{x_1(\sigma): \sigma\in\Lambda_{k,r}\},\;x_2:=\max\{x_2(\sigma): \sigma\in\Lambda_{k,r}\}.
 \]
 The complement of $\bigcup_{\sigma\in\Lambda_{k,r}}J_\sigma$ equals the union of $(-\infty,x_1),(x_2,+\infty)$ and all the possible open intervals between two neighboring cylinders. In each of these open intervals, there are at most three points of $\alpha_n$, otherwise the point in the middle would be redundant, which contradicts the optimality of $\alpha_n$. Also, one can see
 \[
 \alpha\cap(-\infty,x_1)=\emptyset,\;\;\alpha\cap(x_2,+\infty)=\emptyset;
 \]
 otherwise, we may replace $\alpha\cap(-\infty,x_1)$ with $\{x_1\}$, or replace $\alpha\cap(x_2,+\infty)$ with $\{x_2\}$, and get a contradiction.
From the above analysis, we obtain
\[
{\rm card}\bigg(\alpha_n\setminus \bigcup_{\sigma\in\Lambda_{k,r}}J_\sigma\bigg)\leq 2\phi_{k,r}.
\]
Using this and (\ref{n}), we deduce
\begin{eqnarray}\label{pre1}
{\rm card}\bigg(\alpha_n\cap \bigcup_{\sigma\in\Lambda_{k,r}}J_\sigma\bigg)\geq(M_2+2)\phi_{k,r}-2\phi_{k,r}=M_2\phi_{k,r}.
\end{eqnarray}

Suppose $L_\sigma<M_1$ for some $\sigma\in\Lambda_{k,r}$. We deduce a contradiction. By (\ref{pre1}), there exists some $\omega\in\Lambda_{k,r}$ such that
$L_\omega>M_2$. Define
\begin{eqnarray*}
&&\gamma_{L_\omega-M_1-4}(\omega)\in C_{L_\omega-M_1-4,r}(\mu(\cdot|J_\omega)),\;\gamma_{M_1+2}(\sigma)\in C_{M_1+2,r}(\mu(\cdot|J_\sigma));
\\
&&\beta:=(\alpha_n\setminus J_\omega)\cup \gamma_{L_\omega-M_1-4}(\omega)\cup\{x_1(\omega),x_2(\omega)\}\cup \gamma_{M_1+2}(\sigma).
\end{eqnarray*}
Then we have ${\rm card}(\beta)\leq{\rm card}(\alpha)$. For every $\tau\in\Lambda_{k,r}$ with $\tau\neq\sigma,\omega$ and every $x\in J_\tau\setminus (J_\sigma\cup J_\omega)$, we have
$d(x,\beta)\leq d(x,\alpha_n)$. Note that $J_\omega,J_\sigma$ and $J_\tau$ are non-overlapping and $\mu$ is non-atomic. We have $\mu(J_\tau\cap J_\sigma)=0$. It follows that
\begin{eqnarray}\label{t0}
\sum_{\tau\in\Lambda_{k,r},\tau\neq\sigma,\omega}I_\tau(\beta,\mu)\leq\sum_{\tau\in\Lambda_{k,r},\tau\neq\sigma,\omega}I_\tau(\alpha_n,\mu).
\end{eqnarray}
By the supposition, $L_\sigma<M_1$. Then by Remark \ref{remgl} (i) and (iii), we have
\[
I_\sigma(\alpha_n,\mu)\geq \mathcal{E}(\sigma)\max\{e^r_{L_\sigma+2,r}(\nu_\sigma),e^r_{M_1+1,r}(\nu_\sigma)\}\geq \mathcal{E}(\sigma)e^r_{M_1+1,r}(\nu_\sigma).
\]
This, together with the definition of $\beta$, yields
\begin{eqnarray}\label{t1}
I_\sigma(\alpha_n,\mu)-I_\sigma(\beta,\mu)\geq \mathcal{E}(\sigma)(e^r_{M_1+1,r}(\nu_\sigma)-e^r_{M_1+2,r}(\nu_\sigma)).
\end{eqnarray}
Note that $L_\omega>M_2$. By Remark \ref{remgl}, we deduce
\[
I_\omega(\alpha_n,\mu)\geq \mathcal{E}(\omega)e^r_{L_\omega+2,r}(\nu_\omega).
\]
Using this, the definition of $\beta$, Lemma \ref{pre01} and (\ref{lambdakrcompare}), we deduce
\begin{eqnarray}\label{t2}
I_\omega(\beta,\mu)-I_\omega(\alpha_n,\mu)&\leq& \mathcal{E}(\omega)e^r_{L_\omega-M_1-4}(\nu_\omega)-\mathcal{E}(\omega)e^r_{L_\omega+2}(\nu_\omega)\nonumber\\
&<&\mathcal{E}(\omega)\eta_r\big(e^r_{M_1+1,r}(\nu_\sigma)-e^r_{M_1+2,r}(\nu_\sigma)\big)\nonumber\\
&<&\mathcal{E}(\sigma)\big(e^r_{M_1+1,r}(\nu_\sigma)-e^r_{M_1+2,r}(\nu_\sigma)\big).
\end{eqnarray}
By (\ref{t1}), (\ref{t2}), it follows that
\[
I_\omega(\beta,\mu)-I_\omega(\alpha_n,\mu)<I_\sigma(\alpha_n,\mu)-I_\sigma(\beta,\mu).
\]
Combining this with (\ref{t0}), we conclude that
\[
I(\beta,\mu)=\sum_{\tau\in\Lambda_{k,r}}I_\tau(\beta,\mu)<\sum_{\tau\in\Lambda_{k,r}}I_\tau(\alpha_n,\mu)=I(\alpha_n,\mu).
\]
This contradicts the optimality of $\alpha_n$ and the lemma follows.
\end{proof}

Next, we show that for each $\sigma\in\Lambda_{k,r}$ and $\omega\in\Psi_{|\sigma|,3}$, we have $\alpha_n\cap J_{\sigma\ast\omega}\neq\emptyset$. This will be used to give an upper estimate for $\overline{J}(\alpha_n,\mu)$.
\begin{lemma}\label{pre03}
For every $\sigma\in\Lambda_{k,r}$ and $\omega\in\Psi_{|\sigma|,3}$, we have $L_{\sigma\ast\omega}\geq 1$.
\end{lemma}
\begin{proof}
Suppose that $L_{\sigma\ast\omega}=0$ for some $\sigma\in\Lambda_{k,r}$ and $\omega\in\Psi_{|\sigma|,3}$. We deduce a contradiction. By the supposition, we have
$\alpha_n\subset J_{\sigma\ast\omega}^c$. By Lemma \ref{pre01} (1),
\begin{eqnarray}\label{t5}
I_\sigma(\alpha_n,\mu)>\mathcal{E}(\sigma)e^r_{M_1-3,r}(\nu_\sigma).
\end{eqnarray}
On the other hand, by Lemma \ref{pre02}, we know that $L_\sigma\geq M_1$. We set
\[
\gamma_{L_\sigma-2}(\sigma)\in C_{L_\sigma-2,r}(\mu(\cdot|J_\sigma)),\;\beta:=(\alpha_n\setminus J_\sigma)\cup\gamma_{L_\sigma-2}(\sigma)\cup\{x_1(\sigma),x_2(\sigma)\}.
\]
Then we have ${\rm card}(\beta)\leq{\rm card}(\alpha_n)=n$. For every $\tau\in\Lambda_{k,r}\setminus\{\sigma\}$ and $x\in J_\tau\setminus J_\sigma$,
we have $d(x,\beta)\leq d(x,\alpha_n)$. Again, note that $\mu(J_\tau\cap J_\sigma)=0$. It follows that
\begin{eqnarray}\label{t3}
\sum_{\tau\in\Lambda_{k,r},\tau\neq\sigma}I_\tau(\beta,\mu)\leq\sum_{\tau\in\Lambda_{k,r},\tau\neq\sigma}I_\tau(\alpha_n,\mu).
\end{eqnarray}
For $x\in J_\sigma$, we have $d(x,\beta)\leq d(x,\gamma_{L_\sigma-2}(\sigma))$. Hence, by Remark \ref{remgl} (i), (ii),
\begin{eqnarray}\label{t4}
I_\sigma(\beta,\mu)&\leq&\int_{J_\sigma}d(x,\gamma_{L_\sigma-2}(\sigma))^rd\mu(x)\nonumber\\&=&p_\sigma\int_{J_\sigma} d(x,\gamma_{L_\sigma-2}(\sigma))^rd\nu_\sigma\circ g_\sigma^{-1}(x)\nonumber\\&=&p_\sigma c_\sigma^r\int  d(x,g_\sigma^{-1}(\gamma_{L_\sigma-2}(\sigma)))^rd\nu_\sigma(x)\nonumber\\&=&\mathcal{E}(\sigma)e^r_{L_\sigma-2,r}(\nu_\sigma)
\nonumber\\&\leq&\mathcal{E}(\sigma)e^r_{M_1-2,r}(\nu_\sigma),
\end{eqnarray}
where we also used the fact that $L_\sigma\geq M_1$. By (\ref{t5})-(\ref{t4}), we deduce
\[
I(\beta,\mu)=\sum_{\tau\in\Lambda_{k,r}}I_\tau(\beta,\mu)<\sum_{\tau\in\Lambda_{k,r}}I_\tau(\alpha_n,\mu)=I(\alpha_n,\mu).
\]
This contradicts the optimality of $\alpha_n$ and the lemma follows.
\end{proof}

With the next lemma, we establish an upper estimate for $L_\sigma$ with $\sigma\in\Lambda_{k,r}$. This will be used to establish a lower estimate for $\underline{J}(\alpha_n,\mu)$.
\begin{lemma}\label{pre04}
For every $\sigma\in\Lambda_{k,r}$, we have $L_\sigma\leq M_3$.
\end{lemma}
\begin{proof}
Suppose that $L_\sigma>M_3$ for some $\sigma\in\Lambda_{k,r}$. We give the proof by deducing a contradiction. By (\ref{n}) and (\ref{lambdacard}), we have
\[
n<(2+M_2)\phi_{k+1,r}\leq(2+M_2)M_0\phi_{k,r}.
\]
Thus, there exists some $\tau\in\Lambda_{k,r}$ with $L_\tau<(4+M_2)M_0$. Define
\[
\beta:=(\alpha_n\setminus (J_\sigma\cup J_\tau))\cup\{x_1(\sigma),x_2(\sigma)\}\cup\{x_1(\tau),x_2(\tau)\}\cup\gamma_{L_\sigma-7}(\sigma)\cup\gamma_{L_\tau+3}(\tau).
\]
Then we have ${\rm card}(\beta)\leq{\rm card}(\alpha_n)=n$. For every $\tau\in\Lambda_{k,r}\setminus\{\sigma,\tau\}$ and $x\in J_\tau\setminus (J_\sigma\cup J_\tau)$, we have $d(x,\beta)\leq d(x,\alpha_n)$. It follows that
\begin{eqnarray}\label{t6}
\sum_{\rho\in\Lambda_{k,r},\rho\neq\sigma,\tau}I_\tau(\beta,\mu)\leq\sum_{\rho\in\Lambda_{k,r},\rho\neq\sigma,\tau}I_\tau(\alpha_n,\mu).
\end{eqnarray}
For $x\in J_\sigma$, we have $d(x,\beta)\leq d(x,\gamma_{L_\sigma-7}(\sigma))$. By Remark \ref{remgl} (ii),(iii), we deduce
\begin{eqnarray}\label{t7}
I_\sigma(\beta,\mu)-I_\sigma(\alpha_n,\mu)&\leq&\int_{J_\sigma}d(x,\gamma_{L_\sigma-7}(\sigma))^rd\mu(x)-I_\sigma(\alpha_n,\mu)\nonumber\\
&\leq&\mathcal{E}(\sigma)(e^r_{L_\sigma-7,r}(\nu_\sigma)-e^r_{L_\sigma+2,r}(\nu_\sigma)).
\end{eqnarray}
For $x\in J_\tau$, we have $d(x,\beta)\leq d(x,\gamma_{L_\tau+3}(\tau))$. By Remark \ref{remgl} (ii) and (iii), we have
\begin{eqnarray}\label{t8}
I_\tau(\alpha_n,\mu)-I_\tau(\beta,\mu)&\geq&\mathcal{E}(\tau)e^r_{L_\tau+2,r}(\nu_\tau)-\int_{J_\tau}d(x,\gamma_{L_\tau+3}(\tau))^rd\mu(x)\nonumber\\
&=&\mathcal{E}(\tau)(e^r_{L_\tau+2,r}(\nu_\tau)-e^r_{L_\tau+3,r}(\nu_\tau)).
\end{eqnarray}
Combining (\ref{t6})-(\ref{t8}), (\ref{lambdakrcompare}) and Lemma \ref{pre01} (3), we have, $I(\beta,\mu)<I(\alpha_n,\mu)$. This contradicts the optimality of $\alpha_n$ and the proof of the lemma is complete.
\end{proof}

\section{Proof of Theorem \ref{mthm}}

 Let $n$ and $k$ satisfy (\ref{n}) and $\alpha_n\in C_{n,r}(\mu)$. By Lemma \ref{pre03}, for every $\sigma\in\Lambda_{k,r}$ and $\omega\in\Psi_{|\sigma|,3}$, we have $\alpha_n\cap J_{\sigma\ast\omega}\neq\emptyset$. This implies that, for every $a\in\alpha_n$, we have
\begin{eqnarray}\label{sa}
S_a:={\rm card}(\{\sigma\in\Lambda_{k,r}:P_a(\alpha_n)\cap J_\sigma\cap E\neq\emptyset\})\leq 2.
\end{eqnarray}
For every $\rho\in\Lambda_{k,r}$, we write
\begin{eqnarray}
\xi_1(\rho):=\min_{b\in\alpha_n\cap J_\rho}b,\;\;\xi_2(\rho):=\min_{b\in\alpha_n\cap J_\rho,b>\xi_1(\rho)}b;\label{tem03}
\\ \zeta_1(\rho):=\max_{b\in\alpha_n\cap J_\rho}b,\;\;\zeta_2(\rho):=\max_{b\in\alpha_n\cap J_\rho,b<\zeta_1(\rho)}b.\label{tem04}
\end{eqnarray}
Note that $M_1\geq 4$, by Lemmas \ref{pre02} and \ref{pre03}, we have
\[
\xi_1(\rho)<\xi_2(\rho)<\zeta_2(\rho)<\zeta_1(\rho).
\]

Next, we define a set $G_a$ and an auxiliary measure $\lambda_a$ for each $a\in\alpha_n$. We distinguish two cases according to the value of $S_a$.

Case (I): $S_a=2$. In this case, we have $P_a(\alpha_n)\cap E\subset J_\sigma\cup J_\tau$ for some neighboring cylinders $J_\sigma,J_\tau$ with $\sigma,\tau\in\Lambda_{k,r}$. As above, we denote by $x_1(\rho),x_2(\rho)$ the left and right endpoint of the interval $J_\rho$. Without loss of generality, we assume that
\begin{eqnarray}\label{wlg1}
x_2(\sigma)<x_1(\tau),\;\;p_\sigma\geq p_\tau.
\end{eqnarray}
By (\ref{wlg1}), on can see that $\xi_2(\sigma)\leq a\leq\zeta_2(\tau)$. Define
\[
G_a=\bigcup_{b\in\alpha_n\cap[\xi_2(\sigma),\zeta_2(\tau)]}P_b(\alpha_n).
\]
Let $c(\sigma)$ be the midpoint of the interval $[\xi_1(\sigma),\xi_2(\sigma)]$ and $d(\tau)$ the midpoint of $[\zeta_2(\tau),\zeta_1(\tau)]$. Then we have $P_a(\alpha_n)\subset G_a\subset [c(\sigma),d(\tau)]$. By Lemmas \ref{pre02} and \ref{pre04},
\[
2M_1-2\leq T_a:={\rm card}(\alpha_n\cap[\xi_2(\sigma),\zeta_2(\tau)])={\rm card}(\alpha_n\cap G_a)\leq 2M_3.
\]
Let $g_\sigma$ be an arbitrary similitude of similarity ratio $c_\sigma$ on $\mathbb{R}^1$. We define
\begin{eqnarray}\label{z1}
\lambda_a:=\mu(\cdot|G_a)\circ g_\sigma,\; {\rm implying}\;\;\mu(\cdot|G_a)=\lambda_a\circ g_\sigma^{-1}.
\end{eqnarray}
Let $K_a$ denote the support of $\lambda_a$. Then $K_a\subset g_\sigma^{-1}(G_a)$.

Case (II): $S_a=1$. In this case, $P_a(\alpha_n)\cap E\subset J_\sigma$ for some $\sigma\in\Lambda_{k,r}$.  Write
\begin{eqnarray*}
w_1(\sigma):=\min\{a,\xi_2(\sigma)\},\;w_2(\sigma):=\max\{a,\zeta_2(\sigma)\}.
\end{eqnarray*}
Then we have $w_1(\sigma)<w_2(\sigma)$ and $w_1(\sigma)\leq a\leq w_2(\sigma)$. Now we define
\begin{equation}\label{ga2}
G_a:=\bigcup_{b\in\alpha_n\cap[w_1(\sigma),w_2(\sigma)]}P_b(\alpha_n).
\end{equation}
As above, let $g_\sigma$ be a similitude on $\mathbb{R}^1$ with similarity ratio $c_\sigma$. We define
$\lambda_a:=\mu(\cdot|G_a)\circ g_\sigma$ and $K_a:={\rm supp}(\lambda_a)$.
By Lemmas \ref{pre02} and \ref{pre04},
\[
M_1-1\leq T_a:={\rm card}(\alpha_n\cap[w_1(\sigma),w_2(\sigma)])={\rm card}(\alpha_n\cap G_a)\leq M_3.
\]

\begin{remark}\label{zhu3}
The set $G_a$ is the union of at most $2M_3$ elements of the Voronoi partition $\{P_b(\alpha_n)\}_{b\in\alpha_n}$. By Theorem 4.1 in \cite{GL:00}, $\alpha_n\cap G_a\in C_{T_a,r}(\mu(\cdot|G_a))$.
\end{remark}

 In the following, we will assume that $a\in\alpha_n$ and focus on the case that $S_a=2$; the case when $S_a=1$ is relatively much simpler and will be treated as a degenerate case of the former. Assume that $S_a=2$ and (\ref{wlg1}) is satisfied. We set
\begin{equation}\label{tem01}
G_{a,\sigma}:=G_a\cap J_\sigma;\;\;G_{a,\tau}=G_a\cap J_\tau.
\end{equation}
Then we have $\mu(G_{a,\sigma}\cap G_{a,\tau})=0$. Hence,
\begin{equation}\label{s1}
\mu(G_a)=\mu(G_{a,\sigma})+\mu(G_{a,\tau}).
\end{equation}

Besides \cite[Theorem 4.1]{GL:00}, the next lemma will be crucial for our lower estimate for $\underline{J}(\alpha_n,\mu)$. It is a variation of Proposition 12.12 in \cite{GL:00}.
\begin{lemma}\label{pre07}(see \cite[Proposition 12.12]{GL:00})
Let $\nu$ be a Borel probability measure on $\mathbb{R}^q$ with compact support $K_\nu$. Assume that there exist constants $s$ and $\widetilde{C}$, such that
\begin{eqnarray*}\label{s5}
\sup_{x\in\mathbb{R}^q}\nu(B(x,\epsilon))\leq\widetilde{C}\epsilon^s\;{\rm for\;all}\;\;\epsilon>0.
\end{eqnarray*}
then for a Borel set $B\subset\mathbb{R}^q$ and any $b\in\mathbb{R}^q$, we have
\begin{eqnarray*}\label{s6}
\int_Bd(x,b)^rd\mu(x)\geq 2^{-(1+\frac{r}{s})}\widetilde{C}^{-\frac{r}{s}} \nu(B)^{1+\frac{r}{s}}.
\end{eqnarray*}
\end{lemma}
\begin{proof}
It can be proved by the same argument as that for Proposition 12.12 for \cite{GL:00}.
\end{proof}

By the following three lemmas, we explore the local properties of $\lambda_a$ for $a\in\alpha_n$. These properties will enable us to apply Lemma \ref{pre07}.

\begin{lemma}\label{pre05}
Let $t$ be the same as in Lemma \ref{local}. There exists a constant $C_1$ such that
$\sup_{x\in\mathbb{R}^1}\lambda_a(B(x,\epsilon))\leq C_1\epsilon^t$ for all $\epsilon>0$.
\end{lemma}
\begin{proof}
Let $\epsilon>0$ and $x\in\mathbb{R}^1$ be given.
We distinguish two cases.

Case 1: $S_a=2$. Let $\sigma,\tau$ be the same as in (\ref{wlg1}). By Lemma \ref{pre03}, for every $\rho\in\Lambda_{k,r}$ and $\omega\in\Psi_{|\sigma|,3}$, we have ${\rm card}(J_{\rho\ast\omega}\cap\alpha_n)\geq 1$. Thus
\begin{equation}\label{s2}
\mu(G_{a,\sigma})\geq (1-\overline{p}^3)p_\sigma,\;\mu(G_{a,\tau})\geq (1-\overline{p}^3)p_\tau.
\end{equation}
Using (\ref{s1}), (\ref{s2}) and the definition of $\lambda_{a}$, we deduce
\begin{eqnarray*}
\lambda_a(B(x,\epsilon))&=&\mu(\cdot|G_a)\circ g_\sigma(B(x,\epsilon))\\
&=&\frac{1}{\mu(G_a)}\mu(B(g_\sigma(x),c_\sigma\epsilon)\cap G_a)\\
&=&\frac{\mu(B(g_\sigma(x),c_\sigma\epsilon)\cap G_{a,\sigma})+\mu(B(g_\sigma(x),c_\sigma\epsilon)\cap G_{a,\tau})}{\mu(G_{a,\sigma})+\mu(G_{a,\tau})}
\\&\leq&\frac{\mu(B(g_\sigma(x),c_\sigma\epsilon)\cap G_{a,\sigma})}{\mu(G_{a,\sigma})}+\frac{\mu(B(g_\sigma(x),c_\sigma\epsilon)\cap G_{a,\tau})}{\mu(G_{a,\tau})}\\&\leq&\frac{\mu(B(g_\sigma(x),c_\sigma\epsilon)\cap J_\sigma))}{(1-\overline{p}^3)p_\sigma}+\frac{\mu(B(g_\sigma(x),c_\sigma\epsilon)\cap J_\tau))}{(1-\overline{p}^3)p_\tau}
\end{eqnarray*}
This, together with Lemma \ref{local}, yields
\begin{eqnarray}\label{s3}
\lambda_a(B(x,\epsilon))&\leq&\frac{\nu_\sigma\circ g_\sigma^{-1}(B(g_\sigma(x),c_\sigma\epsilon)}{1-\overline{p}^3}+\frac{\nu_\tau\circ g_\tau^{-1}(B(g_\sigma(x),c_\sigma\epsilon))}{1-\overline{p}^3}\nonumber\\&=&\frac{\nu_\sigma (B(x,\epsilon))}{1-\overline{p}^3}+\frac{\nu_\tau(B(g_\tau^{-1}\circ g_\sigma(x),c_\tau^{-1}c_\sigma\epsilon))}{1-\overline{p}^3}\nonumber\\
&\leq& C(1-\overline{p}^3)^{-1}\epsilon^t+C(1-\overline{p}^3)^{-1}(c_\tau^{-1}c_\sigma\epsilon)^t.
\end{eqnarray}
By the definition of $\Lambda_{k,r}$, we have $\mathcal{E}(\tau)\geq \eta_r\mathcal{E}(\sigma)$. By (\ref{wlg1}), we have $p_\sigma\geq p_\tau$. So,
\[
\bigg(\frac{c_\tau}{c_\sigma}\bigg)^r\geq\eta_r\frac{p_\sigma}{p_\tau}\geq\eta_r.
\]
It follows that $c_\tau^{-1}c_\sigma\leq\eta_r^{-1/r}$. Using this and (\ref{s3}), we obtain
\[
\lambda_a(B(x,\epsilon))\leq C(1-\overline{p}^3)^{-1}(1+\eta_r^{-t/r})\epsilon^t=:\chi_1\epsilon^t.
\]

Case 2: $S_a=1$. In this case, $P_a(\alpha_n)\cap E\subset J_\sigma$ for some $\sigma\in\Lambda_{k,r}$ and $|K_a|\leq 1$. Let $G_a$ be as defined in (\ref{ga2}). By Lemma \ref{pre03}, we have $\mu(G_a)\geq p_\sigma(1-2\overline{p}^3)$.
By the definition of $\lambda_{a}$ and Lemma \ref{local}, we deduce
\begin{eqnarray*}
\lambda_a(B(x,\epsilon))&=&\frac{1}{\mu(G_a)}\mu(B(g_\sigma(x),c_\sigma\epsilon)\cap G_a)\\&\leq&\frac{\mu(B(g_\sigma(x),c_\sigma\epsilon)\cap J_\sigma)}{\mu(G_a)}\\&\leq&\frac{\mu(B(g_\sigma(x),c_\sigma\epsilon)\cap J_\sigma))}{(1-2\overline{p}^3)p_\sigma}
=\frac{\nu_\sigma (B(x,\epsilon))}{1-2\overline{p}^3}\\&\leq& C(1-2\overline{p}^3)^{-1}\epsilon^t=:\chi_2\epsilon^t.
\end{eqnarray*}
The lemma follows by setting $C_1:=\max\{\chi_1,\chi_2\}$.
\end{proof}

Assume that $S_a=2$. Let $\sigma,\tau$ be the same as in (\ref{wlg1}). We note that $c_\tau^{-1}c_\sigma$ may be very near zero. This implies that the constant $C_1$ in Lemma \ref{pre05} does not reflect the local property of $\lambda_a(\cdot|g_\sigma^{-1}(G_{a,\tau}))$ in an accurate manner. In order to establish our lower estimate for $\underline{J}(\alpha_n,\mu)$ as claimed in Theorem \ref{mthm}, we need more detailed information on the local behavior of the measure $\lambda_a$. For this purpose, we will consider the performance of $\lambda_a$ on $g_\sigma^{-1}(G_{a,\sigma})$ and $g_\sigma^{-1}(G_{a,\tau})$ separately. Write
\begin{eqnarray}\label{ss4}
&K_{a,\sigma}:=K_a\cap g_\sigma^{-1}(G_{a,\sigma}),\;K_{a,\tau}:=K_a\cap g_\sigma^{-1}(G_{a,\tau});\label{zhu2}
\\ &\lambda_{a,\sigma}:=\lambda_a(\cdot|K_{a,\sigma}),\;\;\lambda_{a,\tau}:=\lambda_a(\cdot|K_{a,\tau}).\label{ss4}
\end{eqnarray}
\begin{lemma}\label{ss5}
Assume that $S_a=2$. Let $K_{a,\sigma},K_{a,\tau}$ be as defined in (\ref{ss4}). There exist constants $D_1$ and $D_2$ such that
\[
\lambda_a(K_{a,\sigma})\geq D_1,\;D_2\frac{c_\sigma^r}{c_\tau^r}\leq\lambda_a(K_{a,\tau})\leq D_2^{-1}\frac{c_\sigma^r}{c_\tau^r}.
\]
\end{lemma}
\begin{proof}
By (\ref{s2}), we have
\begin{eqnarray*}
\frac{(1-\overline{p}^3)p_\sigma}{p_\tau}=\frac{(1-\overline{p}^3)\mu(J_\sigma)}{\mu(J_\tau)}\leq\frac{\mu(G_{a,\sigma})}{\mu(G_{a,\tau})}
\leq\frac{\mu(J_\sigma)}{(1-\overline{p}^3)\mu(J_\tau)}=\frac{p_\sigma}{(1-\overline{p}^3)p_\tau}.
\end{eqnarray*}
Note that, by (\ref{wlg1}), we have $p_\sigma\geq p_\tau$. It follows that
\begin{eqnarray*}
\lambda_a(K_{a,\sigma})=\frac{\mu(G_{a,\sigma})}{\mu(G_{a,\sigma})+\mu(G_{a,\tau})}\geq\frac{1-\overline{p}^3}{2-\overline{p}^3}.\label{g7}
\end{eqnarray*}
So it suffices to set $D_1:=(1-\overline{p}^3)(2-\overline{p}^3)^{-1}$. Similarly, we have
\begin{eqnarray}
\frac{(1-\overline{p}^3)p_\tau}{2p_\sigma}\leq\frac{(1-\overline{p}^3)p_\tau}{(2-\overline{p}^3)p_\sigma}\leq\lambda_a(K_{a,\tau})
=\frac{\mu(G_{a,\tau})}{\mu(G_{a,\sigma})+\mu(G_{a,\tau})}
\leq\frac{p_\tau}{(1-\overline{p}^3)p_\sigma}.\label{g8}
\end{eqnarray}
By the definition of $\Lambda_{k,r}$, we have $\eta_r\mathcal{E}(J_\tau)\leq\mathcal{E}(J_\sigma)\leq\eta_r^{-1}\mathcal{E}(J_\tau)$. Hence,
\begin{eqnarray*}
\frac{\eta_rp_\tau}{p_\sigma}\leq\frac{c_\sigma^r}{c_\tau^r}
\leq\frac{\eta_r^{-1}p_\tau}{p_\sigma}.
\end{eqnarray*}
This, together with (\ref{g8}), implies that
\[
\eta_r(1-\overline{p}^3)\lambda_a(K_{a,\tau})\leq\frac{c_\sigma^r}{c_\tau^r}
\leq2\eta_r^{-1}(1-\overline{p}^3)^{-1}\lambda_a(K_{a,\tau}).
\]
The remaining part of the lemma follows by setting $D_2:=2^{-1}\eta_r(1-\overline{p}^3)$.
\end{proof}
\begin{remark}\label{remcompare}
By Lemma \ref{ss5}, we see that $\mu(G_{a,\sigma})\asymp \mu(G_a)$. It follows that $\mu(G_a)c_\sigma^r\asymp\mathcal{E}(\sigma)$. This justifies the choice $g_\sigma$ in the definition of $\lambda_a=\mu(\cdot|G_a)\circ g_\sigma$.
\end{remark}
\begin{lemma}\label{pre08}
Assume that $S_a=2$. Let $t$ be the same as in Lemma \ref{local} and let $\lambda_{a,\sigma},\lambda_{a,\tau}$ be as defined in (\ref{ss4}). There exists a constant $C_2$ such that
\[
\sup_{x\in\mathbb{R}^1}\lambda_{a,\sigma}(B(x,\epsilon))\leq C_2\epsilon^t\;{\rm and}\;\sup_{x\in\mathbb{R}^1}\lambda_{a,\tau}(B(x,\epsilon))\leq C_2(c_\sigma c_\tau^{-1})^t\epsilon^t.
\]
\end{lemma}
\begin{proof}
Let $x\in\mathbb{R}^1$ and $\epsilon>0$. Note that $\mu(G_a)\lambda_a(K_{a,\sigma})=\mu(G_{a,\sigma})$. We have
\begin{eqnarray*}
\lambda_{a,\sigma}(B(x,\epsilon))&=&\frac{\lambda_a((B(x,\epsilon)\cap K_{a,\sigma}))}{\lambda_a(K_{a,\sigma})}
\\&=&\frac{\mu(g_\sigma((B(x,\epsilon)\cap K_{a,\sigma})\cap G_a)}{\mu(G_a)\lambda_a(K_{a,\sigma})}\\&\leq&\frac{\mu((B(g_\sigma(x),c_\sigma\epsilon)\cap J_\sigma)}{\mu(G_{a,\sigma})}\\&=&\frac{p_\sigma}{\mu(G_{a,\sigma})}\nu_\sigma(B(x,\epsilon)).
\end{eqnarray*}
Thus, by Lemma \ref{local} and (\ref{s2}), we obtain
\[
\lambda_{a,\sigma}(B(x,\epsilon))\leq(1-\overline{p}^3)^{-1} C\epsilon^t.
\]
Set $C_2:=C(1-\overline{p}^3)^{-1}$. Next, we show that the second inequality in the lemma is fulfilled. Let $h_{\sigma,\tau}$ be an arbitrary similitude on $\mathbb{R}^1$ with similarity ratio $c_\sigma c_\tau^{-1}$. we define $\widetilde{\lambda}_{a,\tau}:=\lambda_{a,\tau}\circ h_{\sigma,\tau}^{-1}$. Then $g_\tau:=g_\sigma\circ h_{\sigma,\tau}^{-1}$ is a similitude of similarity ratio $c_\tau$. We have $\widetilde{\lambda}_{a,\tau}=\mu(\cdot|G_{a,\tau})\circ g_\tau$. In fact,
\begin{eqnarray*}
\widetilde{\lambda}_{a,\tau}(B(x,\epsilon))&=&\frac{\lambda_a((B(h_{\sigma,\tau}^{-1}(x),c_\sigma^{-1} c_\tau\epsilon)\cap K_{a,\tau}))}{\lambda_a(K_{a,\tau})}
\\&=&\frac{\mu((B(g_\sigma\circ h_{\sigma,\tau}^{-1}(x),c_\tau\epsilon)\cap G_{a,\tau})}{\mu(G_{a,\tau})}\\&\leq&\frac{\mu((B(g_\tau(x),c_\tau\epsilon)\cap J_\tau)}{\mu(G_{a,\tau})}\\&=&\frac{p_\tau}{\mu(G_{a,\tau})}\nu_\tau(B(x,\epsilon)).
\end{eqnarray*}
By Lemma \ref{local} and (\ref{s2}), we have, $\widetilde{\lambda}_{a,\tau}(B(x,\epsilon))\leq C_2\epsilon^t$. Thus,
\begin{eqnarray*}
\lambda_{a,\tau}(B(x,\epsilon))=\widetilde{\lambda}_{a,\tau}(B( h_{\sigma,\tau}(x),c_\sigma c_\tau^{-1}\epsilon))\leq C_2(c_\sigma c_\tau^{-1}\epsilon)^t.
\end{eqnarray*}
This completes the proof of the lemma.
\end{proof}

Our next lemma gives a lower estimate for $e_{h,r}^r(\lambda_a)-e_{h+1,r}(\lambda_a)$. This estimate will be useful for the application of Lemma 4.1. In \cite[Lemma 2.3]{Zhu:17}, we obtained such a lower estimate by assuming that $|{\rm supp}(\nu)|\leq 1$ for the considered measure $\nu$. However,  as $k\to\infty$, $\sup_{\sigma,\tau\in\Lambda_{k,r}}c_\sigma^{-1}c_\tau$ may not be bounded from above. This means that, in the case that $S_a=2$, $|K_a|$ may be very large, although it is always finite. Note that $|K_{a,\sigma}|\leq 1$ and by Lemma \ref{ss5}, we have $\lambda_a(K_{a,\sigma})\geq  D_1$. This makes the technique in Lemma 2.3 of \cite{Zhu:17} applicable. We have

\begin{lemma}\label{pre06a}
Assume that $S_a=2$ and let $\lambda_a$ be as defined in (\ref{z1}).
Then for each $h\geq 1$, there exists a number $\zeta_{h,r}>0$ depending on $h,t$ and $C_1$, such that
\[
e_{h,r}^r(\lambda_a)-e_{h+1,r}^r(\lambda_a)\geq\zeta_{h,r}.
\]
\end{lemma}
\begin{proof}
Let $h\geq 1$ and $\beta_h\in C_{h,r}(\lambda_a)$. We set
\[
\xi_{h,1}:=(4C_1h)^{-\frac{1}{t}}D_1^{\frac{1}{t}},\;\xi_{h,2}:=(2C_1h)^{-\frac{1}{t}}D_1^{\frac{1}{t}};\;
\delta_h:=\frac{1}{2}\min\{\xi_{h,1},\xi_{h,2}-\xi_{h,1}\}.
\]
Then by Lemma \ref{pre05}, we have $\sum_{b\in\beta}\lambda_a(B(b,\xi_{h,2}))\leq\frac{D_1}{2}$. Note that $\lambda_a(K_{a,\sigma})\geq D_1$ by Lemma \ref{ss5}. It follows that
\begin{eqnarray}\label{g6}
\lambda_a\bigg(K_{a,\sigma}\setminus\bigcup_{b\in\beta_h}B(b,\xi_{h,2})\bigg)\geq D_1-\frac{D_1}{2}=\frac{D_1}{2}.
\end{eqnarray}
Since $|K_{a,\sigma}|\leq1$, we may find an integer $l_h$ which depends on $C_1,t$ and $h$ such that $K_{a,\sigma}\setminus\bigcup_{b\in\beta_h}B(b,\xi_{h,2})$ may be covered by $l_h$ closed balls of radii $\delta_h$ which are centered in $K_{a,\sigma}\setminus\bigcup_{b\in\beta_h}B(b,\xi_{h,2})$. By (\ref{g6}), there exists such a closed ball $B(z_0,\delta_h)$ with $\lambda_a(B(z_0,\delta_h)\cap K_{a,\sigma})\geq (2l_h)^{-1}D_1$. We set $\gamma:=\beta\cup\{z_0\}$. Then
\begin{eqnarray}\label{s4}
e^r_{h,r}(\lambda_a)-e^r_{h+1,r}(\lambda_a)&\geq& I(\beta_h,\mu)-I(\gamma,\mu)\nonumber\\
&\geq&\int_{B(z_0,\delta_h)}d(x,\beta_h)^r-d(x,\gamma)^rd\lambda_a(x)\nonumber\\
&\geq&\lambda_a(B(z_0,\delta_h))(\xi_{h,1}^r-\delta_h^r)\nonumber\\
&\geq&(2l_h)^{-1}D_1\xi_{h,1}^r(1-2^{-r})=:\zeta_{h,r}.
\end{eqnarray}
This completes the proof of the lemma.
\end{proof}
\begin{remark}\label{rem02}
In the cases when $S_a=1$, we have $|K_a|\leq 1$. By using Lemma \ref{pre05}, one can easily see that Lemma \ref{pre06a} holds. See
also \cite[Lemma 2.3]{Zhu:17} for details.
\end{remark}

With the help of Lemmas \ref{pre05}-\ref{pre06a}, we are now able to apply Lemma \ref{pre07} to the measures $\lambda_a$ for $a\in\alpha_n$.
First we consider the case that $S_a=2$. Let $\sigma,\tau$ be the same as in (\ref{wlg1}). We write
\[
\beta_a:=g_\sigma^{-1}(\alpha_n\cap G_a),\;H_a:={\rm card}(\beta_a).
\]
Then by Lemma \ref{pre02} and \ref{pre04}, we have that $2M_1-2\leq H_a\leq 2M_3$. Moreover,
\begin{eqnarray}\label{tem02}
M_1-1\leq{\rm card}(\beta_a\cap g_\sigma^{-1}(G_{a,\sigma})),\;{\rm card}(\beta_a\cap g_\sigma^{-1}(G_{a,\tau}))\leq M_3.
\end{eqnarray}
We denote by $\{P_{b}(\beta_a)\}_{b\in\beta_a}$ a Voronoi partition with respect to $\beta_a$. Then we have
\begin{lemma}\label{pre06}
Assume that $S_a=2$ and let $\lambda_a$ be as defined in (\ref{z1}).
Then there exists a number $d_{H_a}$ depending on $H_a$ and $C_2$ such that
\[
\min_{b\in\beta_a}\int_{P_{b}(\beta_a)}d(x,b)^rd\lambda_a(x)\geq d_{H_a}.
\]
\end{lemma}
\begin{proof}
For convenience, we simply write $H$ for $H_a$. Note that $\alpha_n\in C_{n,r}(\mu)$. By \cite[Theorem 4.1]{GL:00}, $\alpha_n\cap G_a$ is an $H$-optimal set for $\mu(\cdot|G_a)$. Hence, by the similarity of $g_\sigma$ and \cite[Lemma 3.2]{GL:00}, we have $\beta_a\in C_{H,r}(\lambda_a)$. Write
\[
P_{b,1}(\beta_a):=P_b(\beta_a)\cap K_{a,\sigma},\;\;P_{b,2}(\beta_a):=P_b(\beta_a)\cap K_{a,\tau},\;b\in\beta_a.
\]
Fix an arbitrary $b\in\beta_a$, we set $\gamma:=\beta_a\setminus\{b\}$.  Then, by (\ref{tem02}), we have
\[
{\rm card}(\gamma\cap g_\sigma^{-1}(G_{a,\sigma}))\geq M_1-2,\;{\rm card}(\gamma\cap g_\sigma^{-1}(G_{a,\tau}))\geq M_1-2\geq 2.
\]
Note that $|K_{a,\sigma}|\leq 1$ and $|K_{a,\tau}|\leq c_\sigma^{-1}c_\tau$. Hence,
\begin{equation}\label{ss1}
d(x,\gamma)\leq 1,\;{\rm for}\;\;x\in P_{b,1}(\beta_a);\;\;d(x,\gamma)\leq c_\sigma^{-1}c_\tau;\;{\rm for}\;\;x\in P_{b,2}(\beta_a).
\end{equation}
Next, we distinguish two cases according to the following two quantities:
\begin{eqnarray*}
A_{b,1}(\gamma,\lambda_a):=\int_{P_{b,1}(\beta_a)}d(x,\gamma)^rd\lambda_a(x),\\
A_{b,2}(\gamma,\lambda_a):=\int_{P_{b,2}(\beta_a)}d(x,\gamma)^rd\lambda_a(x).
\end{eqnarray*}

Case (1): $A_{b,1}(\gamma,\lambda_a)\geq A_{b,2}(\gamma,\lambda_a)$. In this case, we have
\begin{eqnarray}\label{s4}
e_{H-1,r}^r(\lambda_a)-e_{H,r}^r(\lambda_a)&\leq& I(\gamma,\lambda_a)-I(\beta_a,\lambda_a)\nonumber\\&=&
\int_{P_b(\beta_a)}d(x,\gamma)^rd\lambda_a(x)-\int_{P_b(\beta_a)}d(x,\beta_a)^rd\lambda_a(x)\nonumber\\&\leq& \int_{P_b(\beta_a)}d(x,\gamma)^rd\lambda_a(x)\leq 2A_{b,1}(\gamma,\lambda_a).
\end{eqnarray}
Using this, Lemma \ref{pre06a} and (\ref{ss1}), we deduce
\begin{eqnarray*}
\zeta_{H-1,r}\leq e_{H-1,r}^r(\lambda_a)-e_{H,r}^r(\lambda_a)\leq2A_{b,1}(\gamma,\lambda_a)\leq 2\lambda_a(P_{b,1}(\beta_a)).
\end{eqnarray*}
It follows that $\lambda_a(P_{b,1}(\beta_a))\geq 2^{-1}\zeta_{H,r}$. Note that $\lambda_a(K_{a,\sigma})\geq D_1$.
\begin{eqnarray*}
I_b(\beta_a,\lambda_a)&=&\int_{P_b(\beta_a)}d(x,b)^rd\lambda_a(x)\\
&\geq&\int_{P_{b,1}(\beta_a)}d(x,b)^rd\lambda_a(x)\\
&=&\lambda_a(K_{a,\sigma})\int_{P_{b,1}(\beta_a)}d(x,b)^rd\lambda_a(\cdot|K_{a,\sigma})(x)\\
&\geq&D_1\int_{P_{b,1}(\beta_a)}d(x,b)^rd\lambda_{a,\sigma}(x).
\end{eqnarray*}
Thus, by Lemmas \ref{pre08} and \ref{pre07}, we obtain
\begin{eqnarray}\label{ss2}
I_b(\beta_a,\lambda_a)\geq D_12^{-2(1+\frac{r}{t})}C_2^{-\frac{r}{t}}\zeta_{H-1,r}^{-(1+\frac{r}{t})}.
\end{eqnarray}

Case (2): $A_{b,1}(\gamma,\lambda_a)<A_{b,2}(\gamma,\lambda_a)$. In this case, by (\ref{ss1}), we have
\begin{eqnarray}\label{s4}
\zeta_{H-1,r}\leq e_{H-1,r}^r(\lambda_a)-e_{H,r}^r(\lambda_a)\leq 2(c_\sigma^{-1} c_\tau)^r\lambda_a(P_{b,2}(\beta_a)).
\end{eqnarray}
This and Lemma \ref{ss5} lead to
\[
\lambda_a(P_{b,2}(\beta_a))\geq2^{-1}(c_\sigma c_\tau^{-1})^r\zeta_{H-1,r}\geq2^{-1}D_2\lambda_a(K_{a,\tau})\zeta_{H-1,r}.
\]
It follows that
\begin{eqnarray*}
\lambda_{a,\tau}(P_{b,2}(\beta_a))=\lambda_a(P_{b,2}(\beta_a)|K_{a,\tau})\geq2^{-1}D_2\zeta_{H-1,r}.
\end{eqnarray*}
Let $\widetilde{\lambda}_{a,\tau}$ be as defined in Lemma \ref{pre08}. Then we have
\begin{eqnarray*}
\sup_{x\in\mathbb{R}^1}\widetilde{\lambda}_{a,\tau}(B(x,\epsilon))\leq C_2\epsilon^t;\;\widetilde{\lambda}_{a,\tau}(h_{\sigma,\tau}(P_{b,2}(\beta_a)))\geq2^{-1}D_2\zeta_{H-1,r}.
\end{eqnarray*}
Thus, by Lemma \ref{pre07}, we deduce
\begin{eqnarray*}
&&\int_{h_{\sigma,\tau}(P_{b,2}(\beta_a))}d(x,h_{\sigma,\tau}(b))^rd\widetilde{\lambda}_{a,\tau}(x)\\&&\geq 2^{-(1+\frac{r}{t})}C_2^{\frac{-r}{t}} \widetilde{\lambda}_{a,\tau}(h_{\sigma,\tau}(P_{b,2}(\beta_a)))^{1+\frac{r}{t}}\\&&
\geq2^{-2(1+\frac{r}{t})}C_2^{\frac{-r}{t}}D_2^{1+\frac{r}{t}}\zeta_{H-1,r}^{1+\frac{r}{t}}.
\end{eqnarray*}
This, together with Lemma \ref{ss5} and the similarity of $h_{\sigma,\tau}$, yields
\begin{eqnarray}\label{ss3}
I_b(\beta_a,\lambda_a)&\geq&\int_{P_{b,2}(\beta_a)}d(x,b)^rd\lambda_a(x)\nonumber
\\&=&\lambda_a(K_{a,\tau})\int_{P_{b,1}(\beta_a)}d(x,b)^rd\widetilde{\lambda}_{a,\tau}\circ h_{\sigma,\tau}(x)\nonumber\\
&=&\lambda_a(K_{a,\tau})(c_\sigma^{-1} c_\tau)^r\int_{h_{\sigma,\tau}(P_{b,1}(\beta_a))}d(x,h(b))^rd\widetilde{\lambda}_{a,\tau}(x)
\nonumber\\&\geq&2^{-2(1+\frac{r}{t})}C_2^{\frac{-r}{t}}D_2^{2+\frac{r}{t}}\zeta_{H-1,r}^{1+\frac{r}{t}}.
\end{eqnarray}
The lemma follows by combining (\ref{ss2}) and (\ref{ss3}).
\end{proof}

\begin{remark}\label{rem03}
In case that $S_a=1$, let $G_a$ be as defined in (\ref{ga2}). Then we have $|K_a|\leq 1$. Write
$\beta_a:=g_\sigma^{-1}(\alpha_n\cap G_a)$ and $H:=H_a:={\rm card}(\beta_a)$.
By Lemma \ref{pre05} and Remark \ref{rem02}, one can see that, as a degenerate case, Lemma \ref{pre06} remains true for some number $\widetilde{d}_H$ depending on $t, C_1$ and $H$. One may see \cite[Lemma 2.4]{Zhu:17} for more details. We still denote by $d_H$, the minimum of $d_H$ and $\widetilde{d}_H$. Then Lemma \ref{pre06} holds true for both cases $S_a=2$ and $S_a=1$.
\end{remark}

 For two $\mathbb{R}$-valued variables $X,Y$, we write $X\lesssim Y$ ($X\gtrsim Y$) if there exists some constant $D$ such that $X\leq DY$ ($X\geq DY$).

\emph{Proof of Theorem \ref{mthm}}

Note that $|{\rm supp}(\mu)|\leq 1$. By Lemma \ref{local} and Remarks \ref{rem02}, \ref{rem03}, it suffices to give the proof for $n\geq(M_2+2)\phi_{1,r}$. Let $k$ be as chosen in (\ref{n}) and $\alpha_n\in C_{n,r}(\mu)$. Fix an arbitrary $a\in\alpha_n$, we have
$S_a\leq 2$. By Lemma \ref{pre03}, for every $\sigma\in\Lambda_{k,r}$ and $\omega\in\Psi_{|\sigma|,3}$, we have $\alpha_n\cap J_{\sigma\ast\omega}\neq\emptyset$. This implies that for every $x\in J_\sigma$, we have $d(x,\alpha_n)\leq c_\sigma \overline{c}^3$. Hence,
\begin{eqnarray}\label{z7}
I_a(\mu,\alpha_n)\leq \left\{ \begin{array}{ll}
p_\sigma (c_\sigma \overline{c}^3)^r=\overline{c}^3\mathcal{E}(\sigma)\asymp\eta_r^k&\mbox{if}\;\;S_a=1\\
p_\sigma (c_\sigma \overline{c}^3)^r+p_\tau (c_\tau \overline{c}^3)^r\asymp\eta_r^k&\mbox{if}\;\;S_a=2
\end{array}\right..
\end{eqnarray}

Next we show the reverse estimate. Fix an $a\in\alpha_n$. Let $\beta_a,H_a$ and $\lambda_a$ be as defined above (according to the value of $S_a$).
By Lemmas \ref{pre02} and \ref{pre04}, we have
\begin{eqnarray}\label{z6}
2M_1-2\leq H_a\leq 2M_3.
\end{eqnarray}
By \cite[Theorem 4.1]{GL:00}, $\alpha_n\cap G_a\in C_{T_a,r}(\mu(\cdot|G_a))$. So by the similarity of $g_\sigma$ and Lemma 3.2 of \cite{GL:00}, we know that $\beta_a\in C_{H_a,r}(\lambda_a)$ and $P_{g_\sigma^{-1}(b)}(\beta_a)=g_\sigma^{-1}(P_{b}(\alpha_n))$ for every $b\in \alpha_n\cap G_a$.  Note that $\mu(G_a)\asymp p_\sigma$, in both cases $S_a=1$ and $S_a=2$. Thus, by (\ref{z6}) and Lemmas \ref{pre05}, \ref{pre06} and Remark \ref{rem03}, we deduce
\begin{eqnarray}\label{z2}
I_a(\alpha_n,\mu)&=&\mu(G_a)\int_{P_a(\alpha_n)}d(x,a)^rd\lambda_a\circ g_\sigma^{-1}(x)
\nonumber\\&=&\mu(G_a)c_\sigma^r\int_{g_\sigma^{-1}(P_a(\alpha_n))}d(x,g_\sigma^{-1}(a))^rd\lambda_a(x)\nonumber\\&\gtrsim&\mathcal{E}(\sigma)\min_{1\leq h\leq 2M_3}\zeta_{h,r}\asymp\mathcal{E}(\sigma)\asymp \eta_r^k.
\end{eqnarray}
By \cite[Theorem 4.12]{GL:00}, we have ${\rm card}(\alpha_n)=n$. Thus, by (\ref{z7}) and (\ref{z2}), we obtain
$e^r_{n,r}(\mu)=I(\alpha_n,\mu)\asymp n\eta_r^k$. It follows that $\underline{J}(\alpha_n,\mu),\underline{J}(\alpha_n,\mu)\asymp\frac{1}{n}e^r_{n,r}(\mu)$.

Next, we show the estimate for $\Delta_{n,r}(\mu):=e^r_{n,r}(\mu)-e^r_{n+1,r}(\mu)$.
Fix an arbitrary $\sigma\in\Lambda_{k,r}$. Let $\xi_2(\sigma),\zeta_2(\sigma)$ be as defined in (\ref{tem03}) and (\ref{tem04}). We set
\begin{eqnarray}\label{sg3}
\gamma_\sigma:=\alpha_n\cap[\xi_2(\sigma),\zeta_2(\sigma)],\;T_\sigma:={\rm card}(\gamma_\sigma).
\end{eqnarray}
By Lemmas \ref{pre02}, \ref{pre04}, $2\leq M_1-2\leq T_\sigma\leq M_3$. Choose an arbitrary $a\in\gamma_\sigma$. We have
\begin{eqnarray}\label{sg1}
S_a=1,\;G_a=\bigcup_{b\in\gamma_\sigma}P_b(\alpha_n);\;\mu(G_a)\asymp p_\sigma.
\end{eqnarray}
Let $\Gamma_a\in C_{T_\sigma+1,r}(\mu(\cdot|G_a))$. We define
$\beta:=(\alpha_n\setminus\gamma_\sigma)\cup\Gamma_a$. Then
\begin{eqnarray}\label{sg2}
{\rm card}(\beta)\leq n+1;\;{\rm and}\;d(x,\beta)\leq d(x,\alpha_n)\;\;{\rm for\;all}\;x\in E\setminus G_a.
\end{eqnarray}
Let $g_\sigma$ be a similitude on $\mathbb{R}^1$ of similarity ratio $c_\sigma$ and $\lambda_a:=\mu(\cdot|G_a)\circ g_\sigma^{-1}$ as before. Using (\ref{sg1}), (\ref{sg2}), Remark \ref{remgl} and Lemma \ref{pre06a}, we deduce
\begin{eqnarray}\label{z4}
 \Delta_{n,r}(\mu)&\geq&\int_{G_a}d(x,\alpha_n)^rd\mu(x)-\int_{G_a}d(x,\beta)^rd\mu(x)\nonumber\\
 &\geq&\int_{G_a}d(x,\gamma_\sigma)^rd\mu(x)-\int_{G_a}d(x,\Gamma_a)^rd\mu(x)\nonumber\\&\gtrsim&
 \mathcal{E}(\sigma)(e^r_{T_\sigma,r}(\lambda_a)-e^r_{T_\sigma+1,r}(\lambda_a))\nonumber
\\&\geq&\mathcal{E}(\sigma)\min_{1\leq h\leq M_3}\zeta_{h,r}\gtrsim\frac{1}{n}e^r_{n,r}(\mu).
\end{eqnarray}

Now let $k$ be the integer such that $(M_2+2)\phi_{k,r}\leq n+1<(M_2+2)\phi_{k+1,r}$ and $\alpha_{n+1}\in C_{n+1,r}(\mu)$. Let $\sigma$ be an arbitrary word in $\Lambda_{k,r}$ and let $G_a,\gamma_\sigma$ be as defined in (\ref{sg3}) and (\ref{sg1}) by replacing $\alpha_n$ with $\alpha_{n+1}$. We choose an arbitrary $b\in\gamma_\sigma$ and set $\beta:=\alpha_{n+1}\setminus\{b\}$. For every $x\in E\setminus G_a$, we have $d(x,\beta)=d(x,\alpha_n)$. Note that ${\rm card}(\beta\cap J_\sigma)\geq M_1-1>2$. We have, $d(x,\beta)\leq c_\sigma$ for $x\in G_a$. It follows that
\begin{eqnarray}\label{z5}
 \Delta_{n,r}(\mu)&\leq&\int_{G_a}d(x,\beta)^rd\mu(x)-\int_{G_a}d(x,\alpha_{n+1})^rd\mu(x)\nonumber\\
 &\leq&\int_{G_a}d(x,\beta)^rd\mu(x)\leq \mu(G_a)c_\sigma^r \lesssim\mathcal{E}(\sigma)\lesssim\frac{1}{n}e^r_{n,r}(\mu).
\end{eqnarray}
Hence, the proof of the theorem is complete by combining (\ref{z4}) and ((\ref{z5})).

\end{document}